\documentclass[a4paper,12pt]{amsart}

\usepackage[french, ngerman, italian, english]{babel}
\usepackage{amssymb}
\usepackage{setspace}
\usepackage{anysize}
\usepackage{url}
\usepackage{graphicx}
\usepackage{bm}
\usepackage{hyperref}
\usepackage{mathptmx}
\usepackage{paralist}
\usepackage{verbatim}
\usepackage[utf8]{inputenc}

\theoremstyle{plain}
\newtheorem{thm}{\bf Theorem}[section]

\newtheorem{prop}[thm]{\bf Proposition}
\newtheorem{lemma}[thm]{\bf Lemma}

\newtheorem{question}[thm]{\bf Question}
\newtheorem{problem}[thm]{\bf Problem}

\theoremstyle{definition}
\newtheorem{definition}[thm]{\bf Definition}
\theoremstyle{remark}
\newtheorem{remark}[thm]{\bf Remark}

\theoremstyle{example}

\def \height{{\operatorname{ht}}}

\def\init{\operatorname{in}}

\def \mm{{\mathfrak{m}}}

\def \RR{\mathbb R}

\def \C{\mathcal C}

\def \P{\mathcal P}

\def \NN{{\mathbb{N}}}

\def \QQ{{\mathbb{Q}}}
\def \NN{{\mathbb{N}}}
\def \RR{{\mathbb{R}}}

\def \pp{{\mathfrak{p}}}

\def \s{{\sigma}}

\begin{document}

\title{$F$-thresholds, integral closure and convexity}
\author{Matteo Varbaro}
\address{Dipartimento di Matematica,
Universit\`a degli Studi di Genova, Italy}
\email{varbaro@dima.unige.it}
\dedicatory{To Winfried Bruns on his 70th birthday}
\date{}
\maketitle

\begin{abstract}
The purpose of this note is to revisit the results of \cite{HV} from a slightly different perspective, outlining how, if the integral closures of a finite set of prime ideals abide the expected convexity patterns, then the existence of a peculiar polynomial $f$ allows to compute the $F$-jumping numbers of all the ideals formed by taking sums of products of the original ones. The note concludes with the suggestion of a possible source of examples falling in such a framework.
\end{abstract}

\section{Properties A, A+ and B  for a finite set of prime ideals}

Let $S$ be a standard graded polynomial ring over a field $\Bbbk$ and let $m$ be a positive integer. Fix homogeneous prime ideals of $S$:
\[\pp_1,  \pp_2, \ldots , \pp_m.\]
For any $\sigma=(\sigma_1,\ldots ,\sigma_m)\in\NN^m$ and $k=1,\ldots ,m$, denote by
\[I^{\sigma}:=\pp_1^{\sigma_1}\cdots \pp_m^{\sigma_m} \ \ \ \mbox{ and } \ \ \ e_k(\sigma):=\max\{\ell:I^{\sigma}\subseteq \pp_k^{(\ell)}\}.\] 
Obviously we have $I^{\sigma}\subseteq \bigcap_{k=1}^m\pp_k^{(e_k(\s))}$. Since $S_{\pp_k}$ is a regular local ring with maximal ideal $(\pp_k)_{\pp_k}$, we have that $(\pp_k)_{\pp_k}^{\ell}$ is integrally closed in $S_{\pp_k}$ for any $\ell\in\NN$. Therefore $p_k^{(\ell)}=(\pp_k)_{\pp_k}^{\ell}\cap S$ is integrally closed in $S$ for any $\ell\in\NN$. Eventually we conclude that $\bigcap_{k=1}^m\pp_k^{(e_k(\s))}$ is integrally closed in $S$, so:
\begin{equation}
\overline{I^{\sigma}}\subseteq \bigcap_{k=1}^m\pp_k^{(e_k(\s))}.
\end{equation}

\begin{definition}
We say that $\pp_1,\ldots ,\pp_m$ satisfy condition {\bf A}  if 
\[\overline{I^{\sigma}}= \bigcap_{k=1}^m\pp_k^{(e_k(\s))} \ \ \ \forall \ \sigma\in\NN^m.\]
\end{definition}

If $\Sigma\subseteq \NN^m$, denote by $I(\Sigma):=\sum_{\sigma\in\Sigma}I^{\sigma}$
and by $\overline{\Sigma}\subseteq \QQ^m$ the convex hull of $\Sigma\subseteq \QQ^m$.

\begin{lemma}
For any $\Sigma\subseteq \NN^m$, $\overline{I(\Sigma)}\supseteq \sum_{{\bf v}\in \overline{\Sigma}}I^{\lceil{\bf v}\rceil}$, where $\lceil {\bf v}\rceil:=(\lceil v_1\rceil ,\ldots ,\lceil  v_m\rceil)$ for ${\bf v}=(v_1,\ldots ,v_m)\in\QQ^m$.
\end{lemma}
\begin{proof}
Since $S$ is Noetherian, we can assume that $\Sigma=\{\s^1,\ldots ,\s^N\}$ is a finite set. Take ${\bf v}\in \overline{\Sigma}$. Then there exist nonnegative rational numbers $q_1,\ldots ,q_N$ such that
\[{\bf v}=\sum_{i=1}^Nq_i\sigma^i \ \ \ \mbox{ and } \ \ \ \sum_{i=1}^Nq_i=1.\]
Let $d$ be the product of the denominators of the $q_i$'s and $\s= d\cdot {\bf v}\in\NN^m$. Clearly:
\[(I^{\lceil {\bf v}\rceil})^d=I^{d\cdot \lceil {\bf v}\rceil}\subseteq I^{\s}.\]
Setting $a_i=dq_i$, notice that $\s=\sum_{i=1}^Na_i\sigma^i$ and $\sum_{i=1}^Na_i=d$. Therefore
\[(I^{\lceil {\bf v}\rceil})^d\subseteq I^{\s}\subseteq I(\Sigma)^d.\]
This implies that $I^{\lceil {\bf v}\rceil}$ is contained in the integral closure of $I(\Sigma)$.
\end{proof}

From the above lemma, so, $\overline{I(\Sigma)}\supseteq \sum_{{\bf v}\in\overline{\Sigma}}\overline{I^{\lceil {\bf v}\rceil}}$. In particular: 
\begin{equation}
\mbox{$\pp_1,\ldots ,\pp_m$ satisfy condition {\bf A}} \ \implies \ \overline{I(\Sigma)}\supseteq \sum_{{\bf v}\in\overline{\Sigma}}\left( \bigcap_{k=1}^m\pp_k^{(e_k(\lceil {\bf v}\rceil))}\right).
\end{equation}

\begin{definition}
We say that $\pp_1,\ldots ,\pp_m$ satisfy condition {\bf A+}  if 
\[\overline{I(\Sigma)}= \sum_{{\bf v}\in\overline{\Sigma}}\left( \bigcap_{k=1}^m\pp_k^{(e_k(\lceil {\bf v}\rceil))}\right) \ \ \ \forall \ \Sigma\subseteq \NN^m\]
\end{definition}

\begin{remark}
If $\pp_1,\ldots ,\pp_m$ satisfy condition {\bf A+}, then they satisfy {\bf A} as well (for $\sigma\in\NN^m$, just consider the singleton $\Sigma=\{\sigma\}$).
\end{remark}

\begin{lemma}
Let $\sigma^1,\ldots ,\sigma^N$ be vectors in $\NN^m$, and $a_1,\ldots ,a_N\in \NN$. Then
\[e_k\left(\sum_{i=1}^Na_i\s^i\right)=\sum_{i=1}^Na_ie_k(\s^i) \ \ \ \forall \ k=1,\ldots ,m.\]
\end{lemma}
\begin{proof}
Set $\s=\sum_{i=1}^Na_i\s^i$, and notice that
\[I^{\s}=\prod_{i=1}^N\left(I^{\s^i}\right)^{a_i}\subseteq \prod_{i=1}^N\left(\pp_k^{(e_k(\s^i))}\right)^{a_i}\subseteq \prod_{i=1}^N\pp_k^{(a_ie_k(\s^i))}\subseteq \pp_k^{(\sum_{i=1}^Na_ie_k(\s^i))},\]
so the inequality $e_k(\s)\geq \sum_{i=1}^Na_ie_k(\s^i)$ follows directly from the definition.

For the other inequality, for each $i=1,\ldots ,N$ choose $f_i\in I^{\s^i}$ such that its image in $S_{\pp_k}$ is not in $(\pp_k)_{\pp_k}^{e_k(\s^i)+1}$. Then the class $\overline{f_i}$ is a nonzero element of degree $e_k(\s^i)$ in the associated graded ring $G$ of $S_{\pp_k}$. Being $G$ a polynomial ring (in particular a domain), the element $\prod_{i=1}^N\overline{f_i}^{a_i}$ is a nonzero element of degree
$\sum_{i=1}^Na_ie_k(\s^i)$ in $G$. Therefore 
\[\prod_{i=1}^Nf_i^{a_i}\in I_{\pp_k}^{\s}\setminus \pp_k^{(\sum_{i=1}^Na_ie_k(\s^i)+1)}.\]
This means that $e_k(\s)\leq \sum_{i=1}^Na_ie_k(\s^i)$.
\end{proof}

Consider the function $e:\NN^m\rightarrow \NN^m$ defined by 
\[\sigma\mapsto e(\s):=(e_1(\s),\ldots ,e_m(\s)).\]
From the above lemma we can extend it to a $\QQ$-linear map $e:\QQ^m\rightarrow \QQ^m$. 

Given $\Sigma\subseteq \NN^m$, the above map sends $\overline{\Sigma}$ to the convex hull $P_{\Sigma}\subseteq \QQ^m$ of the set $\{e(\s):\s\in\Sigma\}\subseteq \QQ^m$. In particular we have the following:

\begin{prop}
The prime ideals $\pp_1,\ldots ,\pp_m$ satisfy condition {\bf A+}  if and only if
\[\overline{I(\Sigma)}= \sum_{(v_1,\ldots ,v_m)\in P_{\Sigma}}\left( \bigcap_{k=1}^m\pp_k^{(\lceil v_k\rceil)}\right) \ \ \ \forall \ \Sigma\subseteq \NN^m\]
\end{prop}

If $\Sigma\subseteq \NN^m$ and $s\in\NN$, define $\Sigma^s:=\{\s^{i_1}+\ldots+\s^{i_s}:\s^{i_k}\in\Sigma\}$. Then
\[I(\Sigma)^s=I(\Sigma^s).\]
Furthermore $\overline{\Sigma^s}=s\cdot \overline{\Sigma}$, i.e. $P_{\Sigma^s}=s\cdot P_{\Sigma}$. So:

\begin{prop}\label{crucial}
If $\pp_1,\ldots ,\pp_m$ satisfy condition {\bf A+}, then
\[\overline{I(\Sigma)^s}= \sum_{(v_1,\ldots ,v_m)\in P_{\Sigma}}\left( \bigcap_{k=1}^m\pp_k^{(\lceil sv_k\rceil)}\right) \ \ \ \forall \ \Sigma\subseteq \NN^m, \ s\in \NN.\]
\end{prop}

We conclude this section by stating the following definition:

\begin{definition}
We say that $\pp_1,\ldots ,\pp_m$ satisfy condition {\bf B}  if there exists a polynomial $f\in \bigcap_{k=1}^m\pp_k^{\height(\pp_k)} $ such that $\mathrm{in}_{\prec}(f)$ is a square-free monomial for some term order $\prec$ on $S$. 
\end{definition}

\section{Generalized test ideals and $F$-thresholds}

Let $p>0$ be the characteristic of $\Bbbk$, $I$ be an ideal of $S$ and $\mm$ be the homogeneous maximal ideal of $S$. For all $e\in\NN$, define
\[\nu_e(I):=\max\{r\in\NN:I^r\not\subseteq \mm^{[q]}:=(g^q:g\in \mm)\}, \ \ \ q=p^e.\]
The {\it $F$-pure threshold} of $I$ is then
\[\mathrm{fpt}(I):=\lim_{e\to\infty}\frac{\nu_e(I)}{p^e}.\]
The {\it $p^e$-th root of $I$}, denoted by $I^{[1/p^e]}$, is the smallest ideal $J\subseteq S$ such that $I\subseteq J^{[p^e]}$. By the flatness of the Frobenius over $S$ the $q$-th root is well defined. If $\lambda$ is a positive real number, then it is easy to see that
\[\bigg(I^{\lceil \lambda p^e\rceil}\bigg)^{[1/p^e]}\subseteq \bigg(I^{\lceil \lambda p^{e+1}\rceil}\bigg)^{[1/p^{e+1}]}.\]
The {\it generalized test ideal} of $I$ with coefficient $\lambda$ is defined as:
\[\tau(\lambda\cdot I)\underset{e\gg 0}{:=}\bigg(I^{\lceil \lambda p^e\rceil}\bigg)^{[1/p^e]}.\]
Note that $\tau(\lambda\cdot I)\supseteq \tau(\mu\cdot I)$ whenever $\lambda\leq \mu$. 
By \cite[Corollary 2.16]{BMS}, $\forall \ \lambda\in \RR_{>0}$, $\exists \ \epsilon \in\RR_{>0}$ such that $\tau(\lambda\cdot I)=\tau(\mu\cdot I) \ \ \forall \ \mu\in [\lambda , \lambda+\epsilon)$.
A $\lambda\in \RR_{>0}$ is called an {\it $F$-jumping number} for $I$ if $\tau((\lambda-\varepsilon)\cdot I)\supsetneq \tau(\lambda\cdot I) \ \ \forall \ \epsilon \in \RR_{>0}$.

\bigskip

\bigskip

\bigskip

\setlength{\unitlength}{1cm}
\begin{picture}(15,1.7)(0,0.3)
\put(2,2){\line(1,0){2}}
\put(4,2){\line(1,0){2}}
\put(6,2){\line(1,0){0.6}}
\put(7.3,2){\line(1,0){0.7}}
\put(8,2){\line(1,0){1}}
\put(9.7,2){\line(1,0){1.3}}
\put(1.95,1.9){$[$}
\put(3.87,1.9){$)$}
\put(3.95,1.9){$[$}
\put(5.87,1.9){$)$}
\put(5.95,1.9){$[$}
\put(7.87,1.9){$)$}
\put(7.95,1.9){$[$}
\put(10.9,1.92){$\blacktriangleright$}
\put(11.5,1.9){$\lambda$-axis}
\put(2.3,2.2){$\tau=(1)$}
\put(4.3,2.2){$\tau\neq (1)$}
\put(3.85,1.5){$\lambda_1$}
\put(5.85,1.5){$\lambda_2$}
\put(7.85,1.5){$\lambda_n$}
\put(6.7,2){$\ldots$}
\put(9.1,2){$\ldots$}
\put(3,0.7){$(1)\supsetneq \tau(\lambda_1\cdot I)\supsetneq \tau(\lambda_2\cdot I)\supsetneq \ldots \supsetneq \tau(\lambda_n\cdot I) \supsetneq \ldots$}

\end{picture}

\noindent The $\lambda_i$ above are the $F$-jumping numbers. Notice that $\lambda_1=\mathrm{fpt}(I)$.

\begin{thm}\label{main}
If $\pp_1,\ldots ,\pp_m$ satisfy conditions {\bf A} and {\bf B}, then $\forall \ \lambda\in\RR_{>0}$ we have
\[\tau(\lambda\cdot I^{\sigma})=\bigcap_{k=1}^m\pp_k^{(\lfloor \lambda e_k(\sigma)\rfloor +1-\height(\pp_k))} \ \ \ \ \ \forall \ \sigma\in \NN^m.\] 
If $\pp_1,\ldots ,\pp_m$ satisfy conditions {\bf A+} and {\bf B}, then $\forall \ \lambda\in\RR_{>0}$ we have
\[\tau(\lambda\cdot I(\Sigma))=\sum_{(v_1,\ldots ,v_m)\in P_{\Sigma}}\left(\bigcap_{k=1}^m\pp_k^{(\lfloor \lambda v_k\rfloor +1-\height(\pp_k))}\right) \ \ \ \ \ \forall \ \Sigma\subseteq \NN^m.\] 
\end{thm}

\begin{proof}
The first part immediately follows from \cite[Theorem 3.14]{HV}, for if $\pp_1,\ldots ,\pp_m$ satisfy conditions {\bf A} and {\bf B}, then $I^{\sigma}$ obviously enjoys condition ($\diamond+$) of \cite{HV} $\forall \ \sigma\in \NN^m$. 

Concerning the second part, Proposition \ref{crucial} implies that $I(\Sigma)$ enjoys condition ($*$) of \cite{HV} $\forall \ \Sigma\subseteq \NN^m$ whenever $\pp_1,\ldots ,\pp_m$ satisfy conditions {\bf A+} and {\bf B}. Therefore the conclusion follows once again by \cite[Theorem 4.3]{HV}.
\end{proof}

\section{Where go fishing?}

Let $\Bbbk$ be of characteristic $p>0$. So far we have seen that, if we have graded primes $\pp_1,\ldots ,\pp_m$ of $S$ enjoying {\bf A} and {\bf B}, then we can compute lots of generalized test ideals. If they enjoy {\bf A+} and {\bf B}, we get even more. 

That looks nice, but how can we produce $\pp_1,\ldots ,\pp_m$ like these? Before trying to answer this question, let us notice that, as explained in \cite{HV}, the ideals $\pp_1,\ldots ,\pp_m$ of the following examples satisfy conditions {\bf A+} and {\bf B}:

\begin{itemize}
\item[(i)] $S=\Bbbk[x_1,\ldots ,x_m]$ and $\pp_k=(x_k)$ for all $k=1,\ldots ,m$.
\item[(ii)] $S=\Bbbk[X]$, where $X$ is an $m\times n$ generic matrix (with $m\leq n$) and $\pp_k=I_k(X)$ is the ideal generated by the $k$-minors of $X$ for all $k=1,\ldots ,m$.
\item[(iii)] $S=\Bbbk[Y]$, where $Y$ is an $m\times m$ generic symmetric matrix and $\pp_k=I_k(Y)$ is the ideal generated by the $k$-minors of $Y$ for all $k=1,\ldots ,m$.
\item[(iv)] $S=\Bbbk[Z]$, where $Z$ is a $(2m+1)\times (2m+1)$ generic skew-symmetric matrix and $\pp_k=P_{2k}(Z)$ is the ideal generated by the $2k$-Pfaffians of $Z$ for all $k=1,\ldots ,m$.
\end{itemize}

Even for a simple example like (i), Theorem \ref{main} is interesting: it gives a description of the generalized test ideals of any monomial ideal. 

In my opinion, a class to look at to find new examples might be the following: {\bf fix $f\in S$ a homogeneous polynomial such that $\mathrm{in}_{\prec}(f)$ is a square-free monomial for some term order $\prec$} (better if lexicographical) {\bf  on $S$}, and let $\C_f$ be the set of ideals of $S$ defined, recursively, like follows:

\begin{itemize}
\item[(a)] $(f)\in\C_f$;
\item[(b)] If $I\in\C_f$, then $I:J\in \C_f$ for all $J\subseteq S$;
\item[(c)] If $I,J\in\C_f$, then both $I+J$ and $I\cap J$ belong to $\C_f$.
\end{itemize}

If $f$ is an irreducible polynomial, $\C_f$ consists of only the principal ideal generated by $f$, but otherwise things can get interesting. Let us give two guiding examples:

\begin{itemize}
\item[(i)] If $u:=x_1\cdots x_m$, then the associated primes of $(u)$ are $(x_1),\ldots ,(x_m)$. Furthermore all the ideals of $S=\Bbbk[x_1,\ldots ,x_m]$ generated by variables are sums of the principal ideals above, and all square-free monomial ideals can be obtained by intersecting ideals generated by variables. Therefore, any square-free monomial ideal belongs to $\C_u$, and one can check that indeed:
\[\C_u=\{\mbox{square-free monomial ideals of }S\}.\]
\item[(ii)] Let $X=(x_{ij})$ be an $m\times n$ matrix of variables, with $m\leq n$. For positive integers \ $a_1<\ldots <a_{k} \leq m$ \ and\  $b_1<\ldots <b_{k} \leq n$, \ recall the standard notation for the corresponding $k$-minor:
\[ 
[a_1,\ldots , a_{k}|b_1,\ldots ,b_{k} ]:=\det \left(
\begin{array}{cccc}
x_{a_{1}b_{1}}&x_{a_{1}b_{2}}&\cdots &x_{a_{1}b_{k}}\\
\vdots&\vdots&\ddots&\vdots\\
x_{a_{k}b_{1}}&x_{a_{k}b_{2}}&\cdots &x_{a_{k}b_{k}}\\
\end{array}
\right).
\]
For $i=0,\ldots ,n-m$, let $\delta_i:=[1,\ldots , m|i+1,\ldots ,m+i]$. Also, for $j=1,\ldots ,m-1$ set $g_j:=[j+1,\ldots , m|1,\ldots ,m-j]$ and $h_j:=[1,\ldots ,m-j|n-m+j+1,\ldots ,n]$. 

Let $\Delta$ be the product of the $\delta_i$'s, the $g_j$'s and the $h_j$'s:
\[\Delta:=\prod_{i=0}^{n-m}\delta_i\cdot\prod_{j=1}^{m-1}g_jh_j.\]
By considering the lexicographical term order $\prec$ extending the linear order
\[x_{11}>x_{12}>\ldots x_{1n}>x_{21}>\ldots >x_{2n}>\ldots >x_{m1}>\ldots >x_{mn},\]
we have that 
\[\init(\Delta)=\prod_{i=0}^{n-m}\init(\delta_i)\cdot\prod_{j=1}^{m-1}\init(g_j)\init(h_j)=\prod_{\substack{ i\in\{1,\ldots ,m\} \\ j\in\{1,\ldots ,n\}}}x_{ij}\] 
is a square-free monomial. Since each $(\delta_i)$ belongs to $\C_{\Delta}$, the height-($n-m+1$) complete intersection
\[J:=(\delta_0,\ldots ,\delta_{n-m})\]
is an ideal of $\C_{\Delta}$ too. Notice that the ideal $I_m(X)$ generated by all the maximal minors of $X$ is a height-($n-m+1$) prime ideal containing $J$. So $I_m(X)$ is an associated prime of $J$, and thus an ideal of $\C_{\Delta}$ by definition. With more effort, one should be able to show that the ideals of minors $I_k(X)$ stay in $\C_{\Delta}$ for any size $k$.
\end{itemize}

The ideals of $\C_f$ have quite strong properties. First of all, $\C_f$ is a finite set by \cite{Sc}. Then, all the ideals in $\C_f$ are radical. Even more, Knutson proved in \cite{Kn} that they have a square-free initial ideal!

In order to produce graded prime ideals $\pp_1,\ldots ,\pp_m$ satisfying conditions {\bf A} (or even {\bf A+}) and {\bf B}, it seems natural to seek for them among the prime ideals in $\C_f$. This is because, at least, $f$ is a good candidate for the polynomial needed for condition {\bf B}: if
$f=f_1\cdots f_r$ is the factorization of $f$ in irreducible polynomials, then for each $A\subseteq \{1,\ldots ,r\}$ the ideal
\[J_A:=(f_i:i\in A)\subseteq S\]
is a complete intersection of height $|A|$. If $\pp$ is an associated prime ideal of $J_A$, then $f$ obviously belongs to $\pp^{|A|}\subseteq \pp^{(|A|)}$. So such a $\pp$ satisfies {\bf B}.

\begin{question}
Does the ideal $\pp$ above satisfy condition {\bf A}? Even more, is it true that for prime ideals $\pp$ as above $\pp^s=\pp^{(s)}$ for all $s\in\NN$?
\end{question}

If the above question admitted a positive answer, Theorem \ref{main} would provide the generalized test ideals of $\pp$. A typical example, is when $J_A=(\delta_0,\ldots ,\delta_{n-m})$ and $\pp=I_m(X)$ (see (ii) above), in which case it is well-known that $I_m(X)^s=I_m(X)^{(s)}$ for all $s\in\NN$ (e.g. see \cite[Corollary 9.18]{BV}.

\begin{remark}
Unfortunately, it is not true that $\pp$ satisfies {\bf B} for all prime ideal $\pp\in\C_f$: for example, consider $f=\Delta$ in the case $m=n=2$, that is $\Delta=x_{21}(x_{11}x_{22}-x_{12}x_{21})x_{21}$.
Notice that $(x_{21},x_{11}x_{22}-x_{12}x_{21})=(x_{21},x_{11}x_{22})=(x_{21},x_{11})\cap (x_{21},x_{22})$, so
\[\pp=(x_{21},x_{11})+ (x_{21},x_{22})=(x_{21},x_{11},x_{22})\in \C_{\Delta}.\]
However $\Delta\notin \pp^{(3)}$.
\end{remark}

\begin{problem}
Find a large class of prime ideals in $\C_f$ (or even characterize them) satisfying condition {\bf B}.
\end{problem}

If $\pp_1,\ldots ,\pp_m$ are prime ideals satisfying {\bf A+}, then (by definition)
\[\overline{\sum_{i\in A}\pp_i}=\sum_{i\in A}\pp_i \ \ \ \forall \ A\subseteq \{1,\ldots ,m\}.\]
If $\pp_1,\ldots ,\pp_m$ are in $\C_f$, then the above equality holds true because $\sum_{i\in A}\pp_i$, belonging to $\C_f$, is a radical ideal.

\begin{problem}
Let $\P_f$ be the set of prime ideals in $\C_f$. Is it true that $\P_f$ satisfies condition {\bf A+}? If not, find a large subset of $\P_f$ satisfying condition {\bf A+}.
\end{problem}

\end{document}